\documentclass[preprint, ECP, nobera, nofontexpansion]{ejpecp} 

\SHORTTITLE{Empirical Sphere Concentration}

\TITLE{An Exponential Concentration Inequality for the Components of a Uniform Random Vector on the Sphere 
    } 

\AUTHORS{%
  Joshua~Samani\footnote{University California, Los Angeles, United States of America.
    \EMAIL{jsamani@physics.ucla.edu}}\orcid{0000-0001-8774-6646}
  }


\AMSSUBJ{60E15; 60F10; 60B10; 62G30} 

\SUBMITTED{January 2, 2013} 
\ACCEPTED{December 13, 2014} 

\VOLUME{0}
\YEAR{2023}
\PAPERNUM{0}
\DOI{10.1214/YY-TN}

\ABSTRACT{We show that if $\vec X = (X_1, \dots, X_N)$ is a uniform random vector on the unit Euclidean sphere, the empirical CDF of the components of $\sqrt N \vec X = (\sqrt N X_1, \dots, \sqrt N X_N)$ concentrates exponentially rapidly in $N$ around the standard Gaussian CDF $\Phi$.  More precisely, we find explicit functions $\gamma$ and $g_\pm$ such that the Kolmogorov-Smirnov distance between the empirical CDF of the components of $\sqrt N \vec X$ and $\Phi$ deviates by more than $\epsilon + \gamma(t)$ with probability at most $2e^{-2N\epsilon^2} + e^{-Ng_+(t)^2} + e^{-Ng_-(t)^2}$ for $\epsilon > 0$ and $t\in[0,1)$.  A weaker but more transparent inequality replacing $\gamma$ and $g_\pm$ with linear functions is obtained as a corollary. All functions and constants are explicit, so our bounds offer finite-sample guarantees for statistical applications.}

\newcommand{\reals}{\mathbb R}
\newcommand{\pr}{\mathrm{Pr}}
\DeclareMathOperator{\sgn}{sgn}
\newtheorem*{theorem*}{Theorem}

\begin{document}



\section{Introduction}

Let $\vec X = (X_1, \dots, X_N)$ be a point sampled uniformly at random from the unit sphere $S^{N-1}\subset \reals^N$.  The classical Maxwell-Poincar\'e-Borel lemma asserts that for fixed positive integer $k$, the marginal distribution of $(\sqrt{N} X_1, \dots, \sqrt{N} X_k)$ converges to a $k$-dimensional standard Gaussian distribution as $N \to \infty$. Quantitative versions of this statement were established by Stam \cite{Stam1982} and Diaconis and Freedman \cite{DF1987}, who also traced the history of such results. A generalization to $\ell^p$ spheres appears in Rachev and R\"uschendorf \cite{RR1991}.

Going beyond marginals of subsets of coordinates, Diaconis and Freedman \cite{DF1987} noted that the empirical distribution of all $N$ components of $\sqrt{N} \vec X$ also converges to the standard Gaussian. More recently, Ben Arous, Dembo, and Guionnet \cite[Theorem 6.6]{ADG2001} showed this empirical distribution satisfies a large deviation principle at speed $N$, with an $\ell^p$ generalization by Kim and Ramanan \cite{KR2018}.

We complement these asymptotic results with Theorem \ref{th:main}, an explicit, non-asymptotic concentration inequality for the empirical CDF of $\sqrt{N} \vec X$, bounding Kolmogorov-Smirnov deviations exceeding $\epsilon + \gamma(t)$ with probability at most $2e^{-2N\epsilon^2} + e^{-Ng_+(t)^2} + e^{-Ng_-(t)^2}$, where $\gamma$ and $g_\pm$ are explicit functions. A simpler version with explicit constants is provided in Corollary \ref{cor:simpconc}.  

In Section \ref{sec:main}, we state our main results, and in Section \ref{sec:sketch} we sketch the ideas behind their proofs.  This is followed by key lemmas in Section \ref{sec:lemmas}, proofs of the main results in Section \ref{sec:mainproofs}, and an Appendix detailing longer computations used in the proof of one lemma.

\section{Main Results}\label{sec:main}

The empirical CDF of a random vector $\vec V = (V_1, \dots, V_N)$ whose components $V_i$ are (not necessarily independent) real-valued random variables is defined as
\begin{align}\label{ecdfdef}
    \widehat F_{\vec V}(x) = \frac{1}{N}\sum_{i=1}^N \mathbb{1}_{\{V_i \leq x\}}
\end{align}  
where $\mathbb 1_A$ is the indicator function on a set $A$.  The Kolmogorov-Smirnov (KS) distance between CDFs $F$ and $G$ is defined as
\begin{align}
	d_\mathrm{KS}(F, G) = \sup_{x\in\reals}|F(x) - G(x)|
\end{align}
The following theorem is our main result.

\begin{theorem}[Empirical CDF Concentration]\label{th:main}
	Let $\vec X = (X_1, \dots, X_N)$ be a vector chosen uniformly at random from the unit Euclidean sphere $S^{N-1}$, and let $\Phi$ be the standard Gaussian CDF.  Let $\gamma$ be defined as in Definition \ref{def:deform} below, and let $g_\pm$ be defined as in equation \eqref{gpm} below.  Then for all $\epsilon > 0$ and $t\in[0,1)$,
	\begin{align}\label{mineq}
		\pr\left(d_\mathrm{KS}\left(\widehat  F_{\sqrt{N}\vec X}, \Phi\right) > \epsilon + \gamma(t)\right) \leq 2e^{-2N\epsilon^2} + e^{-Ng_+(t)^2} + e^{-Ng_-(t)^2}
	\end{align}
\end{theorem}
\begin{proof}
	See Section \ref{sec:mainproofs}.
\end{proof}

The meaning of our main inequality \eqref{mineq} is a bit opaque since it depends on the not-so-simple functions $\gamma$ and $g_\pm$, but as we show in Lemma \ref{lem:funcbounds},  these functions satisfy simple bounds that yield  the following weaker but more transparent inequality.

\begin{corollary}[Simplified Concentration]\label{cor:simpconc}
		For all $\epsilon > 0$ and $t \in [0, 1)$,
		\begin{align}\label{explicitbound}
		\pr\left(d_\mathrm{KS}\left(\widehat  F_{\sqrt{N}\vec X}, \Phi\right)> \epsilon + \frac{t}{2}\right) \leq 2e^{-2N\epsilon^2} + e^{-\frac{9}{64}Nt^2} + e^{-Nt^2}
	\end{align}
\end{corollary}

\begin{proof}
	See Section \ref{sec:mainproofs}. 
\end{proof}

\section{Proof Sketch}\label{sec:sketch}

To prove Theorem \ref{th:main}, we use the Gaussian normalization trick which allows us to study the empirical distribution of a rescaled Gaussian random vector $\lambda \vec Z$ instead of the empirical distribution of $\sqrt N \vec X$.  Let $\vec Z = (Z_1, \dots, Z_N)$ be a random vector whose components $Z_i$ are standard Gaussian random variables, and let $|\vec Z| = (Z_1^2 + \cdots + Z_N^2)^{1/2}$ be its Euclidean norm.  Then $\vec X = \vec Z/|\vec Z|$ is uniformly distributed on the unit Euclidean sphere.  Defining $\lambda = \sqrt{N}/|\vec Z|$ gives
\begin{align}\label{coreq}
	\sqrt N \vec X = \lambda \vec Z
\end{align}
Since $\lambda$ concentrates around $1$ as $N\to\infty$, the empirical distribution of $\lambda \vec Z$ approximates that of $\vec Z$.  Moreover, the Dvoretzky-Kiefer-Wolfowitz (DKW) inequality \cite{DKW1956} sharpened by Massart \cite{Massart1990} implies that since the components of $\vec Z$ are i.i.d., the empirical CDF $\widehat F_{\vec Z}$ concentrates within an $\epsilon$ tube centered at the standard Gaussian CDF $\Phi$.
\begin{theorem*}[Dvoretzky-Kiefer-Wolfowitz, 1956]\label{th:DKW}
	Let $\vec V = (V_1, \dots, V_N)$ be a real-valued random vector with i.i.d. components $V_i$ having common CDF $F$, then for all $\epsilon > 0$
	\begin{align}\label{eq:DKW}
    \pr\Big(d_\mathrm{KS}(\widehat F_{\vec V}, F)>\epsilon\Big) \leq 2e^{-2N\epsilon^2}
	\end{align}
\end{theorem*}
Although the DKW inequality places an upper bound on the probability of $\widehat F_{\vec Z}$ escaping the $\epsilon$ tube centered at $\Phi$, rescaling $\vec Z$ by $\lambda$ causes the empirical CDF $\widehat F_{\lambda \vec Z}$ to stretch or compress relative to $\widehat F_{\vec Z}$, so the rescaled CDF $\widehat F_{\lambda \vec Z}$ may escape that tube if $\lambda$ is too far from $1$.  To address this, we restrict attention to scale factors $\lambda$ lying in a small interval $[1-t, 1+t]$, and we ``inflate'' the tube by adding a small quantity $\gamma(t)$ to $\epsilon$.  These two steps ensure that whenever $\widehat F_{\vec Z}$ is contained in the original $\epsilon$ tube, $\widehat F_{\lambda\vec Z}$ will still be contained in the inflated $\epsilon + \gamma(t)$ tube.  To formalize this idea, we define deformed Gaussian CDFs $\Phi^+$ and $\Phi^-$ which inflate the bottom and top of the tube respectively.  The function $\gamma$ is then defined as the largest distance between the original tube and the inflated tube.
  \begin{definition}\label{def:deform}
  	Let $\Phi$ be the standard Gaussian CDF, and let $\sgn$ be the sign function.  For each $t\in[0,1)$, we define
  \begin{align}
	\Phi^\pm_{t}(x) = \Phi\left(\dfrac{x}{1 \mp\sgn(x) t}\right), \qquad
	\gamma(t) = \sup_{x\in\reals}\left(\Phi^+_t(x) - \Phi(x)\right)
\end{align}
  \end{definition}
The approach of restricting $\lambda$ to the interval $[1-t, 1+t]$ costs us a penalty relative to the DKW bound \eqref{eq:DKW}, but because $|\vec Z|^2$ is a $\chi^2$ random variable, the following result of Laurent and Massart \cite[Section 4, Lemma 1]{LM2000} allows us to bound this penalty by a quantity exponentially small in $N$.
\begin{theorem*}[Laurent-Massart, 2000]\label{th:lmcs}
	If $U$ is a $\chi^2$ random variable with $N$ degrees of freedom, then for all $x\geq 0$,
\begin{align}
	\pr\left(U - N \geq 2\sqrt{Nx} + 2x\right)&\leq e^{-x} \label{lmlb}\\
	\pr\left(N -  U\geq 2\sqrt{Nx}\right)&\leq e^{-x} \label{lmub}
\end{align}
\end{theorem*}
In the next section, we prove lemmas that formalize the steps in this sketch, and these lemmas enable short proofs of the main results in the last section.

\section{Lemmas}\label{sec:lemmas}

The lemma formalizing the tube inflation idea of the last section benefits from the observation that although $\gamma$ is defined as the largest difference between $\Phi^+$ and $\Phi$, symmetry allows us to also express it as the largest difference between $\Phi$ and $\Phi^-$.
\begin{lemma}\label{lem:gamalt}
	The function $\gamma$ of Definition \ref{def:deform} can be re-expressed as follows:
	\begin{align}\label{gammdef}
		\gamma(t) = \sup_{x\in\reals}\left(\Phi(x) - \Phi_t^-(x)\right)
	\end{align}
\end{lemma}

\begin{proof}
	Since $\Phi(x) - 1/2$ is odd, $\Phi(x) = 1 -\Phi(-x)$. Thus by Definition \ref{def:deform}, for all $x\geq 0$,
\begin{align}
	\Phi_t^+(x) - \Phi(x)
	&= \Phi\left(\frac{x}{1-t}\right) - \Phi(x)
	= \Phi(-x) - \Phi\left(\frac{-x}{1-t}\right)
	= \Phi(-x) - \Phi_t^-(-x) 
\end{align}
while for all $x\leq 0$
\begin{align}
	\Phi_t^+(x) - \Phi(x)
	&= \Phi\left(\frac{x}{1+t}\right) - \Phi(x) 
	= \Phi(-x) - \Phi\left(\frac{-x}{1+t}\right)
	= \Phi(-x) - \Phi_t^-(-x)
\end{align}
Putting these observations together implies that for all $x\in\reals$
\begin{align}
	\Phi^+_t(x) - \Phi(x) &= \Phi(-x) - \Phi^-_t(-x)
\end{align}
and thus
\begin{align}
	\{\Phi^+_t(x) - \Phi(x)\,|\,x\in\reals\} = \{\Phi(x) - \Phi^-_t(x)\,|\, x\in\reals\}
\end{align}
Taking the supremum of both sides and invoking Definition \ref{def:deform} gives \eqref{gammdef}.
\end{proof}

\begin{lemma}[Tube Inflation]\label{lem:inflate}
	 Let $F:\reals\to [0,1]$ be a CDF, and let $\lambda > 0$, $\epsilon > 0$, and $t\in[0,1)$.  Define the rescaled CDF $F_\lambda(x) = F(x/\lambda)$.  If
\begin{align}
	d_\mathrm{KS}(F, \Phi)\leq \epsilon\quad\text{and}\quad |1-\lambda| \leq t
\end{align}
then
\begin{align}
	d_\mathrm{KS}(F_\lambda, \Phi) \leq \epsilon + \gamma(t)
\end{align}
\end{lemma}

\begin{proof}
Let $|1-\lambda| \leq t$, then if $x\leq 0$, monotonicity of $\Phi$ gives
\begin{align}
	\Phi\left(\frac{x}{1 - t}\right) \leq \Phi\left(\frac{x}{\lambda}\right)\leq  \Phi\left(\frac{x}{1 + t}\right)
\end{align}
while if  $x\geq 0$
\begin{align}
	\Phi\left(\frac{x}{1 + t}\right) \leq \Phi\left(\frac{x}{\lambda}\right)\leq  \Phi\left(\frac{x}{1 - t}\right)
\end{align}
Combining with Definition \ref{def:deform} for $\Phi^\pm$ gives the following for all $x\in\reals$, 
\begin{align}\label{phpm}
	\Phi_t^-(x) \leq \Phi(x/\lambda) \leq \Phi_t^+(x)
\end{align}
But the expressions for $\gamma$ in Definition \ref{def:deform} and equation \eqref{gammdef} imply that for all $x\in\reals$,
\begin{align}\label{gac}
	\Phi(x) - \gamma(t) \leq \Phi_t^-(x), \qquad \Phi_t^+(x) \leq \Phi(x) + \gamma(t) 
	\end{align}
Chaining the inequalities \eqref{phpm} and \eqref{gac} gives
\begin{align}\label{pht}
	\Phi(x) - \gamma(t) \leq \Phi(x/\lambda) \leq \Phi(x) + \gamma(t)
\end{align}
The assumption $d_\mathrm{KS}(F, \Phi)\leq \epsilon$ implies that $\Phi(x) - \epsilon \leq F(x) \leq \Phi(x) + \epsilon$ for all $x\in\reals$, so substituting $x/\lambda$ for $x$, and invoking the definition $F_\lambda(x) = F(x/\lambda)$ gives
	\begin{align}\label{sins}
		\Phi(x/\lambda) - \epsilon \leq F_\lambda(x) \leq \Phi(x/\lambda) + \epsilon
	\end{align}
Combining the inequalities \eqref{pht} and \eqref{sins}, we find that for all $x\in\reals$
\begin{align}\label{sin}
		\Phi(x) - (\epsilon + \gamma(t)) \leq F_\lambda(x) \leq \Phi(x) + (\epsilon + \gamma(t))
	\end{align}
	and therefore $d_\mathrm{KS}(F_\lambda, \Phi) \leq \epsilon + \gamma(t)$.
\end{proof}

The following lemma bounds concentration of the scale factor $\lambda$ around $1$, and it allows us to calculate the penalty we pay by restricting $\lambda$ to lie in the interval $[1-t, 1+t]$.

\begin{lemma}[Concentration of $\lambda$]\label{lem:lamcon}
	Let $\vec Z = (Z_1, \dots, Z_N)$ be a random vector whose components are i.i.d. standard Gaussian random variables, and let $\lambda = \sqrt N/|\vec Z|$.  If we define
	\begin{align}\label{gpm}
		g_+(t) = \frac{1}{2}\left(1-\frac{1}{(1+t)^2}\right), \qquad
		g_-(t) = \frac{1}{2}\left(\sqrt{\frac{2}{(1-t)^2}-1}-1\right)
	\end{align}
	then for all $t\in[0,1)$,
	\begin{align}
		\pr\left(|1-\lambda| > t\right) \leq e^{-Ng_+(t)^2} + e^{-Ng_-(t)^2}
	\end{align}

\end{lemma}

\begin{proof}\label{lamconproof}
	Since the components of $\vec Z$ are independent standard Gaussian random variables, $|\vec Z|^2$ is a $\chi^2$ random variable with $N$ degrees of freedom, so with a little algebra, Theorem \ref{th:lmcs} implies that
	\begin{align}
		\pr\left(|\vec Z|^2> y\right) &\leq \exp\left[-\frac{N}{4}\left(\sqrt{1-2\left(1-\frac{y}{N}\right)} -1\right)^2\right], \qquad y\geq N \label{lbrw}	\\
		\pr\left(|\vec Z|^2 < y\right) &\leq \exp\left[-\frac{N}{4}\left(\frac{y}{N}-1\right)^2\right], \qquad y\leq N \label{ubrw}	
		\end{align}
	We connect these inequalities to concentration of $\lambda$ by noting that
	\begin{align}
		\pr\left(\lambda > 1+t\right)
		&= \pr\left(\frac{\sqrt N}{|\vec Z|} > 1+t\right)
		= \pr\left(\frac{\sqrt N}{1+t} > |\vec Z|\right)
		= \pr\left(|\vec Z|^2 < \frac{N}{(1+t)^2}\right)\\
		\pr\left(\lambda <1-t\right) 
		&= \pr\left(\frac{\sqrt N}{|\vec Z|} < 1-t\right)
		= \pr\left(\frac{\sqrt N}{1-t} < |\vec Z|\right)
		= \pr\left(|\vec Z|^2 > \frac{N}{(1-t)^2}\right)
	\end{align}
	while for every $t\in[0,1)$,
	\begin{align}
		\pr\left(|1-\lambda| > t\right)
		&= \pr\left(\lambda > 1+t\right) + \pr\left(\lambda < 1-t\right) \label{sie}
	\end{align}
	Combining all of these observations yields the following after a bit of algebra:
	\begin{align}
		\pr\left(|1-\lambda| > t\right) 
		&\leq \exp\left[-\frac{N}{4}\left(1-\frac{1}{(1+t)^2}\right)^2\right] + \exp\left[-\frac{N}{4}\left(\sqrt{\frac{2}{(1-t)^2}-1}-1\right)^2\right]
	\end{align}
	Finally invoking the definitions of $g_\pm$ in equation \eqref{gpm} gives the desired result.
\end{proof}

Lemmas \ref{lem:inflate} and \ref{lem:lamcon} are already sufficient for the proof of our main result Theorem \ref{th:main}, but not for its Corollary \ref{cor:simpconc}, so we prove one more lemma which gives bounds on $\gamma$ and $g_\pm$ and thus enables Corollary \ref{cor:simpconc} as a consequence of Theorem \ref{th:main}.

\begin{lemma}\label{lem:funcbounds}
	Let $\gamma$ be defined as in Definition \ref{def:deform}, and let $g_\pm$ be defined as in equation \eqref{gpm}, then for all $t\in[0,1)$,
	\begin{align}
		\gamma(t) \leq t/2, \qquad g_-(t) \geq t, \qquad g_+(t) \geq (3/8)t
	\end{align}
\end{lemma}

\begin{proof}

We start with the bound on $\gamma$.  For each $t\in[0,1)$ define $d_t = \Phi^+_t - \Phi$, then $\gamma(t) = \sup\{d_t(x)\,|\,x\in\reals\}$ by Definition \ref{def:deform}.  Since $d_0 = \Phi_0^+ - \Phi = 0$, we get $\gamma(0) = 0$, and the bound holds at $t = 0$.  For $t\in(0,1)$, we attempt to maximize $d_t$.  We differentiate it for $x\neq 0$ where it is smooth, and we invoke the identity $\Phi'(x) = e^{-x^2/2}/\sqrt{2\pi}$.  This gives
	\begin{align}
		d_t'(x) = 
        	\dfrac{e^{-x^2/2}}{\sqrt{2\pi}} \left(\dfrac{1}{1-\sgn(x)t}\exp\left[-\dfrac{x^2}{2}\left(\dfrac{1}{(1-\sgn(x)t)^2}-1\right)\right] -1\right), \qquad x\neq 0
	\end{align}
	For $t\in(0,1)$, the derivative vanishes at exactly one negative value of $x$ and at exactly one positive value of $x$, which we call $x_-(t)$ and $x_+(t)$ respectively;
	\begin{align}\label{xpmdef}
		x_-(t) = -\sqrt{\frac{2(1+t)^2}{2+t}\cdot \frac{\ln(1+t)}{t}}, \qquad
		x_+(t) = \sqrt{\frac{2(1-t)^2}{2-t}\cdot \frac{\ln(1-t)}{-t}}
	\end{align}
	A bit of computation reveals that $d_t'(x) > 0$ for $x<x_-(t)$ and $d_t'(x) < 0$ for $x_-(t) < x < 0$.  Thus $d_t$ attains a global maximum on $(-\infty, 0)$ at $x_-(t)$.  Similarly $d_t'(x) > 0$ for $0<x<x_+(t)$ while $d_t'(x) < 0$ for $x > x_+(t)$, so $d_t$ attains a global maximum on $(0, \infty)$ at $x_+(t)$.  Since $d_t(0) = 0$, but $d_t(x) > 0$ for $x\neq 0$, one of $x_-(t)$ or $x_+(t)$ (or both) is a global maximum of $d_t$ on $\reals$. Thus if for $t\in(0,1)$ we define $f_-(t) = d_t(x_-(t))$ and $f_+(t) = d_t(x_+(t))$, then
\begin{align}\label{gmax}
	\gamma(t) = \max\{f_-(t), f_+(t)\}
\end{align} 
To further analyze $\gamma$, it is convenient to replace $\ln(1+t)$ with its Taylor series about zero, so the function $\ln(1+t)/t$ becomes smooth on $(-1,1)$ after dividing through by $t$.  This implies that $x_-$ and $x_+$, and by extension $f_-$ and $f_+$, can be regarded as smooth functions on $(-1,1)$.  Recalling Definition \ref{def:deform} and noting that $x_-(t) < 0$ while $x_+(t) > 0$ gives 
\begin{align}\label{fhdefs}
	f_-(t) = \Phi\left(\frac{x_-(t)}{1+t}\right) - \Phi(x_-(t)), \qquad f_+(t) = \Phi\left(\frac{x_+(t)}{1-t}\right) - \Phi(x_+(t))
\end{align}
The computations in Appendix \ref{fmda} show that $f_-'' < 0$, so $f_-$ is concave on $(-1,1)$.  On the other hand, the definitions \eqref{fhdefs} show that $f_+(t) = -f_-(-t)$, giving $f_+'(t) = f_-'(-t)$, and thus $f_+''(t) = -f_-''(-t) > 0$, so $f_+$ is convex.  We recall that a convex (resp. concave) function lies above (resp. below) its tangents.  A short computation shows that $f_+'(0) = f_-'(0) = 1/\sqrt{2e\pi}$, so the functions $f_-$ and $f_+$ share a tangent at the origin, and by convexity, $f_+$ is above that tangent while by concavity, $f_-$ is below it, so $f_+\geq f_-$ on $(-1,1)$.  It follows from $\eqref{gmax}$ that $\gamma = f_+$.  On the other hand, $f_+(t) \to 1/2$ as $t\to 1^-$, so since $f_+(0) = 0$, the line with slope $1/2$ passing through the origin is a secant line of $f_+$, and since $f_+$ is convex, it lies below that secant, which is the desired bound on $\gamma$.

We turn to the bounds on $g_\pm$.  Using definition \eqref{gpm}, $g_-$ is a smooth function on  the interval $(1-\sqrt{2}, 1)$, which contains $0$.   Taking derivatives of $g_-$ gives
	\begin{align}
		g_-'(t) = \frac{1}{(1-t)^2 \sqrt{1+2t-t^2}}, \qquad g_-''(t) = \frac{1 + 6t - 3t^2}{(1-t)^3 \left(1 + 2 t - t^2\right)^{3/2}}
	\end{align}
	It follows that $g_-'(0) = 1$, and inspecting the numerator of $g_-''$, we find that $g_-'' > 0$ for $t\in(1-\sqrt{1+1/3}, 1)$, so $g_-$ is convex on an open interval containing $0$.  Hence on that interval, $g_-$ is bounded below by its tangent at the origin which is the line of unit slope.  In particular $g_-(t)\geq t$ for all $t\in [0,1)$, which is the desired bound on $g_-$.
	Using \eqref{gpm}, we can regard $g_+$ as a smooth function on the interval $(-1, 1)$.  Taking derivatives yields
	\begin{align}
	g_+'(t) = \frac{1}{(1+t)^3}, \qquad g_+''(t) = \frac{-3}{(1+t)^4}
\end{align}
	so $g_+''<0$, and thus $g_+$ is concave on $(-1,1)$.  Since $g_+(0) = 0$ and $g_+(1)\to 3/8$ as $t\to 1^-$, the line passing through the origin with slope $3/8$ is a secant line of $g_+$, and since it is concave, it is bounded below by this line, which is the desired bound on $g_+$
	\end{proof}
	
\begin{remark}\label{rem:constants}
	Since $\gamma = f_+$ is convex on $[0,1)$ and $\gamma'(0) = 1/\sqrt{2\pi e}$, one can choose secants with slopes from $1/2$ down to $1/\sqrt{2\pi e}$ to obtain a tighter bound on $\gamma(t)$ at the cost of restricting the bound to hold on a sub-interval of $[0,1)$. One can thus get as close as desired to $\gamma(t) \leq t/\sqrt{2\pi e}$, and this improves the $t/2$ on the left hand side of inequality \eqref{explicitbound} as close as one wants to $t/\sqrt{2\pi e}$.  Since $g_+$ is concave on $[0,1)$ and $g_+'(0) = 1$, one can choose secants with slopes between $3/8$ and $1$ to get as close as desired to the bound $g_+(t) \geq t$, and this improves the constant $9/64$ as close as desired to $1$ on the right hand side of inequality \eqref{explicitbound}, but again only if one is willing to tolerate a result that holds on a subinterval of $[0,1)$.  The constant in exponent of the last term on the right of \eqref{explicitbound} cannot be improved in the same way.  That constant comes from the lower bound $g_-(t) \geq t$, but $g_-$ is convex on $[0,1)$ while $g_-'(0) = 1$, so $g_-$ is not bounded below by any line passing through the origin with slope greater than $1$.
\end{remark}

\section{Proofs of Theorem \ref{th:main} and Corollary \ref{cor:simpconc}}\label{sec:mainproofs}

\begin{proof}[Proof of Theorem \ref{th:main}]
	The definition \eqref{ecdfdef} of the empirical CDF implies that $\widehat F_{\lambda \vec Z}(x) = \widehat F_{\vec Z}(x/\lambda)$.  Combining this with (the contrapositive of) Lemma \ref{lem:inflate} gives
\begin{align}\label{inclusion}
	\left\{d_\mathrm{KS}(\widehat F_{\lambda \vec Z}, \Phi) > \epsilon + \gamma(t)\right\} \subseteq \left\{d_\mathrm{KS}(\widehat F_{\vec Z}, \Phi)>\epsilon\right\} \cup \Big\{|1-\lambda| > t\Big\}
\end{align}
Using \eqref{coreq} ($\lambda \vec Z = \sqrt{N} \vec X$) and the union bound gives
\begin{align}
		\pr\left(d_\mathrm{KS}\left(\widehat F_{\sqrt{N}\vec X}, \Phi\right)>\epsilon + \gamma(t)\right) 
		&\leq \pr\left(d_\mathrm{KS}(\widehat F_{\vec Z}, \Phi)>\epsilon\right) + \pr\left(|1-\lambda| > t \right)
\end{align}
The DKW inequality \eqref{eq:DKW} bounds the first term on the right, Lemma \ref{lem:lamcon} bounds the second term, and together they yield the desired bound.
\end{proof}

\begin{proof}[Proof of Corollary \ref{cor:simpconc}]
Use the bound on $\gamma$ in Lemma \ref{lem:funcbounds}, then invoke Theorem \ref{th:main}, and finally use the bounds on $g_\pm$ in Lemma \ref{lem:funcbounds} to obtain the following sequence of inequalities.
	\begin{align}
		\pr\left(d_\mathrm{KS}\left(\widehat  F_{\sqrt{N}\vec X}, \Phi\right)> \epsilon + \frac{t}{2}\right) 
		&\leq \pr\left(d_\mathrm{KS}\left(\widehat F_{\sqrt{N}\vec X}, \Phi\right)> \epsilon + \gamma(t)\right)\\
		& \leq 2e^{-2N\epsilon^2} + e^{-Ng_+(t)^2} + e^{-Ng_-(t)^2} \\
		& \leq 2e^{-2N\epsilon^2} + e^{-\frac{9}{64}Nt^2} + e^{-Nt^2}
	\end{align}

\end{proof}

\appendix

\section{Derivatives of $f_-$ and their signs}\label{fmda}

As noted in the proof of Lemma \ref{lem:funcbounds}, we compute the first and second derivatives of $f_-$, and we show that $f_-''<0$ on $(-1,1)$.  It helps to define an auxiliary function
\begin{align}
	\alpha(t) = \frac{1}{2+t} \frac{\ln(1+t)}{t}\label{aldef}
\end{align}
Define $\alpha(0) = 1/2$, and note $\alpha$ extends smoothly to $(-1,1)$ via the Taylor series of $\ln(1+t)$ about zero.  This allows us to regard all functions of $t$ in this appendix as smooth on $(-1,1)$.

\subsection{Computation of $f_-'$ and $f_-''$}

Abbreviating $\alpha = \alpha(t)$, $\alpha' = \alpha'(t)$, recalling \eqref{fhdefs}, and using $\Phi'(x)=\phi(x) = e^{-x^2/2}/\sqrt{2\pi}$ gives
\begin{align}
	f_-'(t) 
	&= \frac{d}{dt}\left[\Phi\left(-\sqrt{2\alpha}\right) -\Phi\left(-(1+t)\sqrt{2\alpha}\right)\right] \\
	&= \phi\left(-\sqrt{2\alpha}\right)\cdot\left(-\frac{\alpha'}{\sqrt{2\alpha}}\right) -\phi\left(-(1+t)\sqrt{2\alpha}\right)\cdot\left(-\sqrt{2\alpha} - \frac{(1+t)\alpha'}{\sqrt{2\alpha}}\right) \label{pmdint}
\end{align}
but
\begin{align}
	\phi(-\sqrt{2\alpha}) 
	= \frac{e^{-\alpha}}{\sqrt{2\pi}} \label{alphA}
\end{align}
and
\begin{align}
	\phi\left(-(1+t)\sqrt{2\alpha}\right)
	&= \frac{1}{\sqrt{2\pi}}\exp\left[-(1+t)^2\alpha\right] \\
	&= \frac{1}{\sqrt{2\pi}}\exp\left[-\alpha\cdot t(2+t) - \alpha\right] \\
	&= \frac{1}{\sqrt{2\pi}}\exp\left[\ln\left(\frac{1}{1+t}\right) - \alpha\right] \\
	&= \frac{e^{-\alpha}}{\sqrt{2\pi}}\frac{1}{1+t} \label{alphB}
\end{align}
Combining \eqref{pmdint}, \eqref{alphA}, and \eqref{alphB} gives
\begin{align}
	f_-'(t)
	&=\frac{e^{-\alpha}}{\sqrt{2\pi}} \left(-\frac{\alpha'}{\sqrt{2\alpha}}\right) -\frac{e^{-\alpha}}{\sqrt{2\pi}}\frac{1}{1+t}\left(-\sqrt{2\alpha} - \frac{(1+t)\alpha'}{\sqrt{2\alpha}}\right)
	= \frac{e^{-\alpha}}{\sqrt{2\pi}}\frac{\sqrt{2\alpha}}{1+t} \label{fmd} 
\end{align}
Taking another derivative gives
\begin{align}
	f_-''(t)
	&=
	-\frac{e^{-\alpha}\sqrt{2\alpha}}{\sqrt{2\pi}}\frac{1}{(1+t)^2}
	+ \frac{-e^{-\alpha}\alpha'\sqrt{2\alpha} + e^{-\alpha}(\sqrt{2\alpha})^{-1}\alpha'}{\sqrt{2\pi}} \frac{1}{1+t} \\
	&= -\frac{f_-'(t)}{1+t}\left[1 - (1+t)\alpha'\left(\frac{1}{2\alpha}-1\right)\right]\label{fmdd}
\end{align}
We now analyze the terms in $f_-''$ to determine its sign.  Inspection of \eqref{aldef} reveals that $\alpha > 0$, so $1/(2\alpha) > 0$ and thus the factor $1/(2\alpha)-1>-1$ in the expression for $f_-''$ in \eqref{fmdd}.  On the other hand, we show in the subsections below that $0<-(1+t)\alpha'<1$, so the product $-(1+t)\alpha'(1/(2\alpha)-1)$ is greater than $-1$, and thus the term in brackets in the expression \eqref{fmdd} for $f_-''$ is positive on $(-1,1)$.  But inspection of \eqref{fmd} shows that $f_-'>0$, and $(1+t)$ is also positive.  Putting this all together, and referring to \eqref{fmdd}, we find that $f_-''<0$ as desired.

\subsection{Proof that $0<-(1+t)\alpha'$}

Referring again to \eqref{aldef}, we compute
\begin{align}
	-(1+t)\alpha' 
	&= -(1+t)\left[-\frac{1}{(2+t)^2}\frac{\ln(1+t)}{t} + \frac{1}{(2+t)}\frac{t(1+t)^{-1} - \ln(1+t)}{t^2}\right] \\
	&= \frac{1+t}{2+t}\alpha + \frac{(1+t)\ln(1+t)-t}{t^2(2+t)} \label{aint}
\end{align}
This expression evaluates to $1/2$ at $t=0$, so we need only consider $t\in(-1,1)$ with $t\neq 0$.  Since $\alpha > 0$, the first term is positive on $(-1,1)$.  Since $(1+t)\ln(1+t)$ is convex and is tangent to the line with unit slope intersecting the origin, we get $(1+t)\ln(1+t) - t \geq 0$, so the second term in \eqref{aint} is non-negative, and thus $0<-(1+t)\alpha'$ on $(-1,1)$.

\subsection{Proof that $-(1+t)\alpha'<1$}

Since \eqref{aint} evaluates to $1/2$ at $t=0$, we need only consider $t\in(-1,1)$ with $t\neq 0$.  Starting from \eqref{aint}, we re-write $-(1+t)\alpha'$ in the following way:
\begin{align}
	-(1+t)\alpha' 
	&= \frac{1+t}{2+t}\alpha + \frac{(1+t)\ln(1+t)-t}{t^2(2+t)} \\
	&= \frac{(1+t)\ln(1+t)}{(2+t)^2t}+ \frac{(1+t)\ln(1+t)-t}{t^2(2+t)} \\
	&= \frac{2(1+t)^2\ln(1+t)- t(2+t)}{t^2(2+t)^2} \\
	&= \frac{2(1+t)^2\ln(1+t)- t(2+t) - t^2(2+t)^2}{t^2(2+t)^2} + 1\\
	&= - \frac{(1+t)^2\left[t(2+t)-2\ln(1+t)\right]}{t^2(2+t)^2} + 1
\end{align}
But $t(2+t)$ is convex while $2\ln(1+t)$ is concave, and they share a tangent at the origin, so $t(2+t) > 2\ln(1+t)$ for $t\neq 0$.  Therefore the term in brackets is positive for nonzero $t\in(-1,1)$, and thus $-(1+t)\alpha' < 1$.







\begin{acks}
I am grateful to Alec Stein for his infinite capacity to discuss ``the piston'' which led me to think about the ideas in this paper and to Julian Gold for his immensely helpful comments on the paper's structure. 
\end{acks}


\end{document}